\RequirePackage{fix-cm}
\documentclass{svjour3}                     
\smartqed  
\usepackage{graphicx}
\usepackage{amsmath,amsfonts,amssymb}
\usepackage{algpseudocode}
\usepackage[ruled]{algorithm}
\usepackage{psfrag}
\usepackage{graphicx}
\usepackage{subfig}
\usepackage{color}
\usepackage{xcolor}

\newcommand{\imbb}[1]{{\color{black}#1}}
\newcommand{\av}{\mathsf a}
\newcommand{\bb}{\mathsf b}
\newcommand{\cc}{\mathsf c}
\newcommand{\dd}{\mathsf d}
\newcommand{\e}{\mathsf e}

\newcommand{\g}{\mathsf g}

\newcommand{\oo}{\mathsf o}
\newcommand{\p}{\mathsf p}
\newcommand{\q}{\mathsf q}
\newcommand{\rr}{\mathsf r}
\newcommand{\s}{\mathsf s}

\newcommand{\uu}{\mathsf u}
\newcommand{\vv}{\mathsf v}
\newcommand{\w}{\mathsf w}
\newcommand{\x}{\mathsf x}
\newcommand{\y}{\mathsf y}
\newcommand{\z}{\mathsf z}
\newcommand{\eps}{\varepsilon}
\newcommand{\Ab}{\mathsf A}

\newcommand{\Db}{\mathsf D}

\newcommand{\Ib}{\mathsf I}

\newcommand{\Mb}{\mathsf M}

\newcommand{\Ob}{\mathsf O}
\newcommand{\Pb}{\mathsf P}
\newcommand{\Qb}{\mathsf Q}

\newcommand{\Ub}{\mathsf U}
\newcommand{\Vb}{\mathsf V}

\newcommand{\Xb}{\mathsf X}

\newcommand{\R}{\mathbb{R}}
\newcommand{\OM}{C}

\newcommand{\LMO}{\tx{LMO}}
\DeclareMathOperator*{\argmax}{arg\,max}
\DeclareMathOperator*{\argmin}{arg\,min}
\DeclareMathOperator*{\secondmax}{secondmax}
\newcommand{\ps}[2]{#1^\intercal #2}
\newcommand{\tx}[1]{\textnormal{#1}}
\newcommand{\n}[1]{\| #1 \|}
\newcommand{\G}{\mathcal{G}}

\newcommand{\F}{\mathcal{F}}

\newcommand{\sm}{\setminus}
\newcommand{\irg}[2]{[{#1}\! : \! {#2}]}
\newcommand{\T}{ ^{\intercal}}
\DeclareMathOperator*{\rank}{rank}
\DeclareMathOperator*{\psd}{psd}
\DeclareMathOperator*{\conv}{conv}
\DeclareMathOperator*{\sign}{sign}
\DeclareMathOperator{\dist}{dist}
\DeclareMathOperator{\aff}{aff}
\DeclareMathOperator{\diam}{diam}
\DeclareMathOperator{\supp}{supp}
\DeclareMathOperator{\intp}{int}
\DeclareMathOperator{\ri}{ri}
\DeclareMathOperator{\PWidth}{PWidth}
\DeclareMathOperator{\faces}{faces}
\DeclareMathOperator{\tr}{tr}

\begin{document}
	
	\title{Frank-Wolfe and friends:\\[0.2em] a journey into projection-free\\[0.2em] \imbb{first-order} optimization \imbb{methods}}
	
	\titlerunning{Frank-Wolfe and friends}        
	
	\author{Immanuel. M.~Bomze         \and
		Francesco~Rinaldi \and 
		Damiano~Zeffiro
	}
	
	\authorrunning{Immanuel~M.~Bomze, Francesco~Rinaldi and Damiano~Zeffiro} 
	
	\institute{Immanuel~M.~Bomze \at
		ISOR, VCOR \& ds:UniVie, Universit\"{a}t Wien, Austria\\
		\email{immanuel.bomze@univie.ac.at}  
		\and
		Francesco~Rinaldi \at
		Dipartimento di Matematica ``Tullio Levi-Civita'', Universit\`a
		di Padova, Italy 	\\
		\email{rinaldi@math.unipd.it}
		\and
		Damiano~Zeffiro \at
		Dipartimento di Matematica ``Tullio Levi-Civita'', Universit\`a
		di Padova, Italy 		\\
		\email{damiano.zeffiro@math.unipd.it}
	}
	
	\date{Received: date / Accepted: date}

	\maketitle
	
	\begin{abstract} Invented some 65 years ago in a seminal paper by Marguerite Straus-Frank and Philip Wolfe, the Frank-Wolfe method recently enjoys a remarkable revival, fuelled by the need of fast and reliable first-order optimization methods in Data Science and other relevant application areas. This review tries to explain the success of this approach by illustrating versatility and applicability in a wide range of contexts, combined with an account on recent progress in variants, both improving on the speed and efficiency of this surprisingly simple principle of first-order optimization.
	
		\keywords{ First-order \imbb{methods} \and Projection-free methods  \and Structured optimization \and  Conditional gradient \and Sparse optimization}
	\end{abstract}
	\clearpage
	\section{Introduction}
	
	In their seminal work \cite{frank1956algorithm}, Marguerite Straus-Frank and Philip Wolfe introduced a first-order algorithm  for the minimization of convex quadratic objectives over polytopes, now known as Frank-Wolfe (FW) method. The main idea of the method is \imbb{simple:} to generate a sequence of feasible iterates by moving at every step towards a minimizer of a linearized objective, the so-called FW vertex.
	Subsequent works, partly motivated by applications in optimal control theory (see \cite{dunn1979rates} for references), generalized the method to smooth \imbb{(possibly non-convex)} optimization over closed subsets of Banach spaces admitting a linear minimization oracle (see \cite{demianov1970approximate,dunn1978conditional}). 
	
	Furthermore, while the ${\mathcal O}(1/k)$ rate in the original article was proved to be optimal when the solution lies on the boundary of the feasible set \cite{canon1968tight}, improved rates were given in a variety of different settings. In \cite{levitin1966constrained} and \cite{demianov1970approximate}, a linear convergence rate was proved over strongly convex domains assuming a lower bound on the gradient norm, a result then extended in \cite{dunn1979rates} under more general gradient inequalities.   In \cite{guelat1986some}, linear convergence of the method was proved for strongly convex objectives with the minimum \imbb{obtained} in the relative interior of the feasible set. 
	
	The slow convergence behaviour for objectives with solution on the boundary motivated the introduction of several variants, the most popular being Wolfe's away step \cite{wolfe1970convergence}.
	Wolfe's idea was to move away from bad vertices, in case a step of the FW method moving towards good vertices did not lead to sufficient improvement on the objective.   This idea was successfully applied in several network equilibrium problems, where linear minimization can be achieved by solving a min-cost flow problem (see \cite{fukushima1984modified} and references therein). In \cite{guelat1986some}, some ideas already sketched by Wolfe were formalized to prove linear convergence of the Wolfe's away step method  and  identification of the
	face containing the solution \imbb{in finite time, under some suitable strict complementarity assumptions.} 
	
	In recent years, the FW method has regained popularity thanks to its ability \imbb{to handle the structured constraints appearing in machine learning and data science applications efficiently.} Examples include LASSO, SVM training,  matrix completion, minimum enclosing ball, density mixture estimation, cluster detection, to name just a few (see Section \ref{Sec:examples} for further details). 
	
	One of the main features of the FW algorithm is its ability to naturally identify  sparse and structured (approximate) solutions. For instance, if the optimization domain is the simplex, then after $k$ steps the cardinality of the support of the last iterate generated by the method is at most $k + 1$. Most importantly, in this setting every vertex added to the support at every iteration must be the best possible in some sense, a property that connects the  method with many greedy optimization schemes \cite{clarkson2010coresets}.  This makes the FW method pretty eff\imbb{icient} on the \imbb{abovementioned problem} class. Indeed,
	the combination of structured solutions with often noisy data makes the sparse approximations found by the method possibly more desirable than high precision solutions generated by a faster converging approach. In some cases, like in cluster detection (see, e.g., \cite{bomze1997evolution}), finding the support of the solution is actually enough to solve the problem independently from the precision achieved.  
	
	Another important feature is that the linear minimization used in the method is often cheaper than the projections required by projected-gradient methods. It is important to notice that, even when these two operations have the same complexity, constants defining the related bounds can differ significantly (see \cite{combettes2021complexity} for some examples and tests). When dealing with large scale problems, the FW method hence has a much smaller per-iteration cost with respect to projected-gradient methods. For this reason, FW methods fall into the category of {\em projection-free methods}~\cite{lan2020first}. Furthermore, the method can be used to approximately solve quadratic subproblems in accelerated schemes, an approach usually referred to as conditional gradient sliding (see, e.g., \cite{carderera2020second,lan2016conditional}). 
	
\subsection{\imbb{Organisation of the paper}	}
	The \imbb{present review is not intended to provide an exhaustive literature survey, but rather as an advanced tutorial demonstrating versatility and power of this approach. The article} is structured as follows: in Section \ref{s:pgs}, we introduce the classic FW method, together with a general scheme for all the methods we consider. In Section \ref{Sec:examples}, we present applications from classic optimization to more recent machine learning problems. In Section \ref{s:stp}, we review some important stepsizes for first order methods. In Section \ref{s:cFW}, we discuss the main theoretical results about the FW method and the most popular variants, including the ${\mathcal O}(1/k)$ convergence rate for convex objectives, affine invariance, the sparse approximation property, and support identification. In Section \ref{s:ir} we illustrate some recent improvements on the ${\mathcal O}(1/k)$ convergence rate. 
	Finally, in Section \ref{s:E} we present recent FW variants fitting different optimization frameworks, in particular block coordinate, distributed, accelerated, and trace norm optimization.

\subsection{\imbb{Notation} 	}
\imbb{For any integers $a$ and $b$, denote by $\irg ab = \{ x \mbox{ integer}: a\le x\le b\}$ the integer range between them. For a set $V$, the power set $2^V$ denotes the system all subsets of $V$, whereas for any positive integer $s\in \mathbb{N}$ we set
${V\choose s} := \{ S\in 2^V : |S| = s\}$,
with $|S|$ denoting the number of elements in $S$. Matrices are denoted by capital sans-serif letters (e.g., the zero matrix $\Ob$, or the $n\times n$ identity matrix $\Ib_n$ with columns $\e_i$ the length of which should be clear from the context). The all-ones vector is $\e:=\sum_i \e_i\in \R^n$. Generally, vectors are always denoted by boldface sans-serif letters $\x$, and their transpose
by $\x\T$. The Euclidean norm of $\x$ is then $\n{\x} := \sqrt{\x\T\x}$ whereas the general $p$-norm is denoted by ${\n{\x}}_p$ for any $p\ge 1$ (so  ${\n{\x}}_2=\n{\x}$). By contrast, the so-called zero-norm simply counts the number of nonzero entries:
$${\n{\x}}_0 := |\{ i\in \irg 1n : x_i\neq 0\}|\, .$$
For a vector $\dd$ we denote as $\widehat{\dd} :=\frac 1{\n{\dd}}\,\dd$ its normalization, with the convention $\widehat{\dd} = \oo$ if $\dd = \oo$. Here $\oo$ denotes the zero vector. In context of symmetric matrices, ``$\psd$'' abbreviates ``positive-semidefinite''.}

	\section{Problem and general scheme} \label{s:pgs}
	We consider the following problem: 
	\begin{equation} \label{fxOM}
		\min_{\x \in \OM} f(\x)
	\end{equation}
	where $\OM$ is a convex and \imbb{compact} (i.e. bounded and closed) subset of $\R^n$ and, unless specified otherwise, $f$ is a differentiable function having Lipschitz continuous gradient with constant $L>0$: $$\n{\nabla f(\x) - \nabla f(\y)} \leq L \n{\x - \y}\quad \imbb{\mbox{for all }\{\x,\y\} \subset \OM\, .}$$
	\imbb{Throughout the article}, we denote by $\x^*$ a \imbb{(global)} solution to \eqref{fxOM} and use the symbol $f^*:=~f(\x^*)$ as a shorthand for the corresponding optimal value. 
	
	The general scheme of the first-order methods we consider for problem \eqref{fxOM}, reported in Algorithm \ref{alg:GS}, \imbb{is based upon a set $F(\x,\g)$ of directions feasible at $\x$ using first-order local information on $f$ around $\x$, in the smooth case $\g=\nabla f(\x)$. From this set, a particular $\dd\in F(\x,\g)$ is selected}, with the maximal stepsize $\alpha^{\max}$ possibly dependent from auxiliary information available to the method \imbb{(at iteration $k$, we thus write $\alpha^{\max}_k$)}, and not always equal to the maximal feasible stepsize.
	
		\begin{algorithm}[H]
			\caption{\texttt{First-order method}}
			\label{alg:GS}
			\begin{algorithmic}
				\par\vspace*{0.1cm}
				\item$\,\,\,1$\hspace*{0.1truecm} Choose a point $x_0\in \Omega$
				\item$\,\,\,2$\hspace*{0.1truecm} For $k=0,\ldots$
				\item$\,\,\,3$\hspace*{0.9truecm} If $\x_k$ satisfies some specific condition, then STOP 
				\item$\,\,\,4$\hspace*{0.9truecm} Choose $\dd_k \in F(\x_k, \nabla f(\x_k))$
				\item$\,\,\,5$\hspace*{0.9truecm} Set $\x_{k+1}=\x_k+\alpha_k \dd_k$, with $\alpha_k\in (0,\alpha^{\max}_k]$ a suitably chosen stepsize
				\item$\,\,\,6$\hspace*{0.1truecm} End for
				\par\vspace*{0.1cm}
			\end{algorithmic}
		\end{algorithm}

	\subsection{The classical Frank-Wolfe method}
	The classical FW method for minimization of a smooth objective $f$ generates a sequence of feasible points $\{\x_k\}$ following the scheme of Algorithm \ref{alg:FW}. At the iteration $k$ it moves toward a vertex \imbb{i.e., an extreme point,} of the feasible set minimizing the scalar product with the current gradient $\nabla f(\x_k)$. It therefore  makes use of a linear minimization oracle (LMO) for the feasible set $\OM$ 
	\begin{equation}\label{eq:linsubp}
	\imbb{	\LMO_{\OM}(\g) \in \argmin_{\z \in \OM} \ps{\g}{\z} }\, ,
	\end{equation} 
	defining the descent direction as
	\begin{equation} \label{fwdk}
		\dd_k = \dd_k^{FW} := \s_k - \x_k, \ \ \s_k \in \LMO_{\OM}(\nabla f(x_k)) \, .
	\end{equation}
	In particular, the update at step 6 can be written as
	\begin{equation}
		\x_{k + 1} = \x_k + \alpha_k (\s_k - \x_k) = \alpha_k \s_k + (1 - \alpha_k) \x_k
	\end{equation}
	Since $\alpha_k \in [0, 1]$, by induction $\x_{k + 1}$ can be written as a convex combination of elements in the set $S_{k + 1} := \{\x_0\} \cup \{\s_i\}_{0 \leq i \leq k}$. When $C = \conv(A)$ for a set $A$ of points with some common property, usually called "elementary atoms", if $x_0 \in A$ then $x_{k}$ can be written as a convex combination of $k + 1$ elements in $A$. \imbb{Note that due to Caratheodory's theorem, we can even limit the number of occurring atoms to $\min \{k,n\}+1$.} In
	the rest of the paper the primal gap at iteration $k$ is defined as $h_k=f(\x_k)-f^*$.
	
\begin{algorithm}[H]
			\caption{\texttt{Frank-Wolfe method}}
			\label{alg:FW}
			\begin{algorithmic}
				\par\vspace*{0.1cm}
				\item$\,\,\,1$\hspace*{0.1truecm} Choose a point $x_0\in \Omega$
				\item$\,\,\,2$\hspace*{0.1truecm} For $k=0,\ldots$
				\item$\,\,\,3$\hspace*{0.9truecm} If $\x_k$ satisfies some specific condition, then STOP 
				\item$\,\,\,4$\hspace*{0.9truecm} Compute $\s_k \in \LMO_{\OM}(\nabla f(x_k))$
				\item$\,\,\,5$\hspace*{0.9truecm} Set $\dd_k^{FW} =  \s_k - \x_k$ \ \ 
				\item$\,\,\,6$\hspace*{0.9truecm} Set $\x_{k+1}=\x_k+\alpha_k \dd_k^{FW}$, with $\alpha_k\in (0,1]$ a suitably chosen stepsize
				\item$\,\,\,7$\hspace*{0.1truecm} End for
				\par\vspace*{0.1cm}
			\end{algorithmic}
\end{algorithm}

	\section{Examples}\label{Sec:examples}
	
	FW methods and variants are a natural choice for constrained optimization on convex sets admitting a linear minimization oracle significantly faster than computing a projection. We present here in particular the traffic assignment problem, submodular optimization, LASSO problem, matrix completion, adversarial attacks, minimum enclosing ball, SVM training, maximal clique search in graphs.

	\subsection{Traffic assignment}
	
	Finding a traffic pattern satisfying the equilibrium conditions in a transportation network is a classic problem in optimization that dates back 
	to Wardrop's paper \cite{wardrop1952road}. Let $\G$ be a network with set of nodes $\irg 1n$. Let $\{D(i, j)\}_{i \neq j}$ be demand \imbb{coefficients}, modeling the amount of goods with destination $j$ and origin $i$. \imbb{For any $i,j$ with $i\neq j$ let furthermore $f_{ij}: \R \to \R$ be non-linear} cost functions, and $x_{ij}^s$ be the flow on  link $(i, j)$ with destination $s$. The traffic assignment problem can be modeled as the following non-linear \textit{multicommodity network} problem \cite{fukushima1984modified}:
	\begin{equation} \label{P:mcn}
		\min \left\{\sum_{i, j} f_{ij}\left( \sum_s x_{ij}^s\right) : \sum_{i}x^s_{i\ell} - \sum_{j} x_{\ell j}^s = D(\ell, s) \, , \mbox{ all }  \ell \neq s, \, \, x_{ij}^s \geq 0  \right\} \, . 
	\end{equation}
	Then the linearized optimization subproblem necessary to compute the FW vertex can be split in $n$ shortest paths subproblems, each of the form
	\begin{equation}
		\min \left\{\sum_{i,j}c_{ij}x_{ij}^s : \sum_{i} x_{i\ell}^s - \sum_{j} x_{\ell j}^s = D(\ell, s), \, \ell \neq s \right\}
	\end{equation}
	for a fixed $s \in \irg 1n$. A number of FW variants were proposed in the literature for efficiently handling this kind of problems (see, e.g., \cite{bertsekas2015convex,fukushima1984modified,leblanc1975efficient,weintraub1985accelerating} and references therein for further details). In the more recent work \cite{joulin2014efficient} a FW variant also solving a shortest path subproblem at each iteration was applied to image and video co-localization. 
	
	\subsection{Submodular optimization}
	
	Given a finite set $V$, a function $r: 2^V \rightarrow \mathbb{R}$ is said to be submodular if for every $A, B \subset V$
	\begin{equation}
	r(A) + r(B) \geq r(A \cup B) + r(A \cap B) \, .
	\end{equation}
	As is common practice in the optimization literature (see e.g. \cite[Section 2.1]{bach2013learning}), here we always assume $s(\emptyset) = 0$.
	A number of machine learning problems, including image segmentation and
	sensor placement, can be cast as minimization of a submodular function (see, e.g., \cite{bach2013learning,chakrabarty2014provable} and references therein for further details):
	\begin{equation} \label{subopt}
		\min_{A \subseteq V} r(A)\, .
	\end{equation}
	Submodular 	optimization  can also be seen
	as a more general way to relate combinatorial problems
	to convexity, 
	for example for structured sparsity~\cite{bach2013learning,jaggi2013revisiting}. By a theorem from \cite{fujishige1980lexicographically}, problem \eqref{subopt} can be in turn reduced to an minimum norm point problem over the base polytope 
	\begin{equation}
		B(G) = \{\s \in \R^V  :   \sum_{a \in A} s_a \leq r(A) \mbox{ for all }  A \subseteq V\, , \,  \sum_{a \in V} s_a = r(V) \} \, .
	\end{equation}
	For this polytope, linear optimization can be achieved with a simple greedy algorithm. More precisely, consider the LP $$\max_{\s \in B(F)} \w\T\s \, .$$ Then if the objective vecor $\w$ has a negative component, the problem is clearly unbounded. Otherwise, a solution to the LP  can be obtained by ordering $\w$ in decreasing manner as $w_{j_1} \geq w_{j_2} \geq ... \geq w_{j_n}$, and setting
	\begin{equation}
		s_{j_k} := r(\{j_{1}, ..., j_{k} \}) - r(\{j_1, ..., j_{k - 1}\}) \, ,
	\end{equation}
	for $k \in \irg 1n$. We thus have a LMO with a $\mathcal{O}(n\log n)$ cost. This is the reason why FW variants are widely used in the context of submodular optimization; further details can be found in, e.g.,  \cite{bach2013learning,jaggi2013revisiting}. 
	
	\subsection{LASSO problem}
	The LASSO, proposed by Tibshirani in 1996 \cite{tibshirani1996regression}, is a popular tool for sparse linear regression. 
	Given the training set
	$$
	T=\{(\rr_i,b_i) \in \R^n\times\R :  i\in \irg 1m\}\, ,
	$$
	where $\rr_i\T$ are the rows of an $m\times n$
matrix $\Ab$,	the goal is  finding a sparse linear model (i.e., a model with a small number of non-zero parameters) describing the data. This problem is strictly 
	connected with the Basis Pursuit Denoising \imbb{(BPD)} problem in signal analysis (see, e.g., \cite{chen2001atomic}). In this case,
	given a discrete-time input signal $b$, and a \textit{dictionary}
	$$\{\av_j\in \R^m \ : \ j\in  \irg 1n \}$$
	of elementary discrete-time signals, usually called atoms (here $\av_j$ are the columns of a matrix $\Ab$), the 
	goal is finding a sparse linear combination of the atoms that {\em approximate} the real signal. From a purely formal point of view, LASSO and BPD problems are equivalent, and both can be formulated as follows:
	\begin{equation}
		\begin{array}{ll}
			\displaystyle{\min_{\x \in \R^n}}&f(\x):=\|\Ab\x-\bb\|_2^2\\
			s.t. & \|\x\|_1\leq \tau\, ,
		\end{array}
	\end{equation}
	where the parameter $\tau$ controls the amount of shrinkage that is
	applied to the model (\imbb{related to sparsity, i.e., the} number of nonzero components in $\x$). The feasible set is
	$$C=\{\x \in \R^n :   \|\x\|_1\leq \tau \}=\conv \{\pm \tau \e_i : \ i\in \irg 1n \}\, .$$
	Thus we have the following LMO in this case:
	$$\LMO_C(\nabla f(\x_k))=\sign(-\nabla_{i_k} f(\x_k))\cdot \tau \e_{i_k}\, ,$$ 
	with $i_k  \in \displaystyle\argmax_{i} |\nabla_i f(\x_k)|$. It is easy to see that the FW per-iteration cost is then $\mathcal{O}(n)$. The peculiar structure of the problem makes FW variants well-suited for its solution. This is the reason why LASSO/BPD problems were considered in a number of FW-related papers (see, e.g., \cite{jaggi2011sparse,jaggi2013revisiting,lacoste2015global,locatello2017unified}).

	\subsection{Matrix completion}		\label{s:mcp}
	Matrix completion is a widely studied problem that comes up in many areas of science and engineering, including collaborative filtering, machine learning, control, remote sensing, and computer vision (just to name a few; see also \cite{candes2009exact} and references therein). The goal is to retrieve a low rank matrix $\Xb \in \R^{n_1 \times n_2}$ from a sparse set of observed matrix entries $\{U_{ij}\}_{(i,j) \in J}$ with $J \subset \irg 1{n_1} \times \irg 1{n_2}$. Thus the problem can be formulated as follows \cite{freund2017extended}:
	\begin{equation} \label{eq:mc}
		\begin{array}{ll}
			\displaystyle{\min_{\Xb \in \R^{n_1 \times n_2}}}&f(\Xb) := \displaystyle\sum_{(i,j)\in J} (X_{ij} - U_{ij})^2\\
			\quad s.t. & \rank(\Xb)\leq \delta,
		\end{array}
	\end{equation}
	where the function $f$ is given by the squared loss over the observed entries of the matrix and $\delta>0$ is a parameter representing the assumed belief about the rank of the reconstructed matrix we want to get in the end.	In practice, the low rank constraint is relaxed with a nuclear norm ball constraint, where we recall that the 
	\imbb{nuclear norm  ${\n{\Xb}}_*$ of a matrix $\Xb$} is equal the sum of its singular values. Thus we get the following convex optimization problem:
	\begin{equation}
		\begin{array}{ll}
			\displaystyle{\min_{\Xb \in \R^{n_1 \times n_2}}}& \displaystyle\sum_{(i,j)\in J}(X_{ij} - U_{ij})^2\\
			\quad s.t. & {\n{\Xb}}_*\leq \delta\, .
		\end{array}
	\end{equation}
	The feasible set is the convex hull of rank-one matrices:
	\begin{equation*}
	    \begin{array}{rcl}
C &= &\{\Xb \in \R^{n_1 \times n_2} :   {\|\Xb\|}_*\leq \delta \}\\[0.3em]
&=& \conv \{\delta \uu\vv\T :\uu\in\R^{n_1},\vv\in\R^{n_2},\  \|\uu\|=  \|\vv\|=1 \} 	\, .\end{array}\end{equation*}
	If we indicate with $\Ab_J$ the matrix that coincides with $\Ab$ on the indices $J$ and is zero otherwise, then we can write $\nabla f(\Xb)=\imbb{2}\,(\Xb-\Ub)_J$. Thus we have the following LMO in this case:   
	\begin{equation}
		\LMO_C(\nabla f(\Xb_k)) \in \argmin \{\tr(\nabla f(\Xb_k)\T \Xb) : {\n{\Xb}}_* \leq \delta\}\, ,
	\end{equation}
	which  boils down to computing the gradient, and the  rank-one matrix $\delta \uu_1 \vv_1\T$, with $\uu_1, \vv_1$ right and left singular vectors corresponding to the top singular value of $-\nabla f(\Xb_k)$.  Consequently, the FW method at a given iteration approximately reconstructs the target matrix as a sparse combination of rank-1 matrices. Furthermore, as the gradient matrix is sparse (it  only has $|J|$  non-zero entries)  storage and approximate singular vector computations can be performed much more efficiently than for dense matrices\footnote{Details related to the LMO cost can be found in, e.g., \cite{jaggi2013revisiting}.}. A number of FW variants has hence been proposed in the literature for solving this problem (see, e.g., \cite{freund2017extended,jaggi2011sparse,jaggi2013revisiting}).
	
	\subsection{Adversarial attacks in machine learning}
	
	Adversarial examples are  maliciously perturbed inputs designed to mislead a properly trained learning machine at test time. An \textit{adversarial attack} hence consists in taking a correctly classified data point $x_0$ and slightly modifying it to create a new data point that leads the considered model to misclassification (see, e.g., \cite{carlini2017towards,chen2017zoo,goodfellow2014generative} for further details). A possible formulation of the problem (see, e.g., \cite{chen2020frank,goodfellow2014generative}) is given by the so called \textit{maximum allowable $\ell_p$-norm attack}  that is,
	\begin{equation}\label{prob_bb_attack}
		\begin{split}
			& \min_{\x\in \R^n} \, f(\x_0+\x) \\
			& s.t. \quad {\|\x\|}_p \le \varepsilon \, ,
		\end{split}
	\end{equation}
	where $f$ is a suitably chosen attack loss function, $\x_0$ is a correctly classified data point, $\x$ represents the additive noise/perturbation, $\varepsilon > 0$ denotes the magnitude of the attack, and $p\geq 1$. It is easy to see that the LMO has a cost $\mathcal{O}(n)$. If $\x_0$ is
	a feature vector of a dog image correctly classified by our learning machine, our adversarial attack hence suitably perturbs the feature vector (using the noise vector $\x$),  thus getting a new feature vector $\x_0+\x$ 
	classified, e.g., as a cat. In case a target adversarial class is specified by the attacker, we have a \emph{targeted attack}.
	In some scenarios, the goal may not be to push $\x_0$ to a specific target class, but rather push
	it away from its original class. In this case we have a so called \emph{untargeted attack}. The attack function $f$ will hence be chosen depending on the kind of attack we aim to perform over the considered model. Due to its specific structure, problem \eqref{prob_bb_attack} can be nicely handled by means of tailored FW variants.  Some FW  frameworks for adversarial attacks were recently described in, e.g.,  \cite{chen2020frank,kazemi2021generating,sahu2020decentralized}.

	\subsection{Minimum enclosing ball}
	Given a set of points $P = \{\p_1,\ldots, \p_n\}\subset\R^d$, the minimum enclosing ball problem (MEB, see, e.g.,  \cite{clarkson2010coresets,yildirim2008two}) consists in finding the \imbb{smallest ball containing $P$.} Such a problem models numerous important applications in clustering, nearest
	neighbor search, data classification, machine learning, facility
	location, collision detection, and computer graphics, \imbb{to name just a few}.
	We refer the reader to \cite{kumar2003approximate}  and the references therein for further details. \imbb{Denoting by $\cc\in \R^d$ the center and by $\sqrt{\gamma}$ (with $\gamma\ge 0$) the radius of the ball, a} convex quadratic formulation for this problem is
	\begin{alignat}{2}
		&\min_{(\cc,\gamma) \in \R^d\times\R} && \, \gamma \\
		&\quad\ s.t.  && \,  \n{\p_i - \cc}^2 \leq \gamma \, , \; \tx{ all } i \in \irg 1n\, .
	\end{alignat}
	This problem can be formulated via \imbb{Lagrangian} duality as \imbb{a convex {\em Standard Quadratic Optimization Problem} (StQP, see, e.g.~\cite{BomdeK2002})}
	\begin{equation} \label{eq:MEBd}
	\imbb{	\min\left\{\x^{\intercal}\Ab^{\intercal}\Ab\x - \bb\T\x :  \x \in \Delta_{n - 1} \right\} }
	\end{equation}
	with $\Ab = [\p_1, ..., \p_n]$ and $\bb\T = [\p_1^{\intercal}\p_1, \ldots , \p_n^{\intercal}\p_n]$. The feasible set is \imbb{ the standard simplex}
	$$\Delta_{n - 1}:=\{\x \in \R^n_+ :  \e^\intercal \x=1\}=\conv\{\e_i : i\in \irg 1n \}\, ,$$
	and the LMO is defined as follows: 
	$$\LMO_{\Delta_{n - 1}}(\nabla f(\x_k))=\e_{i_k},$$  
	with $i_k \in \argmin_{i} \nabla_i f(\x_k)$. 
	It is easy to see that cost  per iteration is $\mathcal{O}(n)$. When applied to \eqref{eq:MEBd}, the FW method can find an $\varepsilon$-cluster in ${\mathcal O}(\frac{1}{\varepsilon})$,  where an $\varepsilon$-cluster is a subset $P'$ of $P$ such that the MEB of $P'$ dilated by $1 + \varepsilon$ contains \imbb{$P$} \cite{clarkson2010coresets}. The set $P'$ is given by the atoms in $P$ selected by the LMO in the first ${\mathcal O}(\frac{1}{\varepsilon})$ iterations. Further details related to the connections between FW methods and MEB problems can be found in, e.g., \cite{ahipacsaouglu2013modified,clarkson2010coresets,damla2008linear} and references therein.
	
	\subsection{Training linear Support Vector Machines}
	{\em Support Vector Machines  (SVMs)} represent a very important class of machine learning tools (see, e.g., \cite{vapnik2013nature} for further details). Given a labeled set of data points, usually called \emph{training set}:  $$TS=\{(\p_i, y_i),\ \p_i \in \R^d,\ y_i \in \{-1, 1\},\  i=1,\dots,n \},$$  the linear SVM training problem consists in finding a linear classifier $\w \in \R^d$ such that the label $y_i$ can be deduced with the "highest possible confidence" from $\w^{\intercal}\p_i$. A convex quadratic formulation for this problem is the following \cite{clarkson2010coresets}:
	
	\begin{equation}\label{p:svm}
		\begin{array}{cl}
			\displaystyle{\min_{\w\in \R^d, \rho \in \R}}&\rho + \frac{\n{\w}^2}{2}\\
			s.t. & \rho + y_i\, \ps{\w}{\p_i} \geq 0\, , \quad \tx{all }i\in \irg 1n\, ,
		\end{array}
	\end{equation}
	where the slack variable $\rho$ stands for the negative margin and we can have $\rho < 0$ if and only if there exists an exact linear classifier, i.e. $\w$ such that $\ps{\w}{\p_i} = \sign(y_i)$. The dual of \eqref{p:svm} is \imbb{again an StQP}:
	\begin{equation} \label{P:SVM}
	\imbb{	\min\left\{\x^{\intercal}\Ab^{\intercal}\Ab\x  :  \x \in \Delta_{n - 1} \right\} } 	\end{equation}
	with $\Ab = [y_1\p_1, ..., y_n\p_n]$. Notice that problem \eqref{P:SVM} is equivalent to an MNP problem on \imbb{$ \conv\{ y_i\p_i : i\in \irg 1n\}$, see Section~\ref{mnpsec} below}. Further details on FW methods for SVM training problems can be found in, e.g., \cite{clarkson2010coresets,jaggi2011sparse}. 
	
	\subsection{Finding maximal cliques in graphs} 
	In the context of network analysis the clique model, dating back at least to the work of Luce and Perry \cite{luce1949method} about social networks, refers to subsets with every two elements in a direct relationship. The problem of finding maximal cliques has numerous applications in domains including telecommunication networks, biochemistry, financial networks, and scheduling (see, e.g., \cite{bomze1999maximum,wu2015review}). Let $G=(V,E)$ be a simple undirected graph with $V$ and $E$ set of vertices and edges, respectively. A clique in $G$ is a subset $C\subseteq V$ such that $(i,j)\in E$ for each  $(i,j)\in C$, with $i\neq j$. The goal in finding a clique  $C$ such that $|C|$ is maximal (i.e., it is not contained in any strictly larger clique). This corresponds to find a local minimum for the  following equivalent 
	\imbb{(this time non-convex) StQP}
	(see, e.g., 
	\cite{bomze1997evolution,bomze1999maximum,hungerford2019general} for further details):
	\begin{equation}\label{eq:MaxClique}
		\max\left\{ \x^{\intercal}\Ab_G \x+\frac{1}{2}\|\x\|^2 :  \x \in \Delta_{n - 1} \right\} 	\end{equation}
	where $\Ab_G$ is the adjacency matrix of $G$. Due to the peculiar structure of the problem, FW methods can be fruitfully used to find maximal cliques (see, e.g., \cite{hungerford2019general}). 

	\section{Stepsizes} \label{s:stp}
	Popular rules for determining the stepsize are:
	
	\begin{itemize}
		\item \emph{diminishing stepsize}: 
		\begin{equation} \label{alpha2k2}
			\alpha_k = \frac{2}{ k + 2} \, ,
		\end{equation}
		mainly used for the classic FW (see, e.g., \cite{freund2016new,jaggi2013revisiting});
		\item  \emph{exact line search}: 
		\begin{equation}
			\alpha_k = \imbb\min \imbb{\argmin_{\alpha \in [0, \alpha_k^{\max}]} \varphi(\alpha)} \quad\imbb{\mbox{with }\varphi (\alpha):=f(\x_k + \alpha \, \dd_k) }\, ,
		\end{equation}
		where we pick the \imbb{smallest} minimizer of the function $\varphi$ \imbb{for the sake of being well-defined even in rare cases of ties} 	(see, e.g., \cite{bomze2020active,lacoste2015global});
		\item 	  \emph{Armijo line search:} 
		the method iteratively shrinks the step size in order to guarantee a sufficient reduction of the objective function. It represents a good way to replace exact line search in cases when it gets too costly. In practice, 
		we fix parameters $\delta\in (0,1)$  and $\gamma\in(0,\frac 12)$,   then try steps $\alpha=\delta^{m}\alpha_k^{\max}$ 
		with $m\in \{ 0,1,2,\dots\}$ until the sufficient decrease inequality 
		\begin{equation}\label{sdcond}
			f(\x_k+\alpha \,\dd_k)\leq f(\x_k)+\gamma \alpha\, \nabla f(\x_k)^{\intercal}  \dd_k
		\end{equation}
		holds, and set $\alpha_k=\alpha$ (see, e.g., \cite{bomze2019first} and references therein).
		\item 	  \emph{Lipschitz constant dependent step size:} 
		\begin{equation} \label{eq:L}
			\alpha_k = \alpha_k(L) := \min \left\{ -\, \frac{ \nabla f(\x_k)\T\dd_k}{L\n{\dd_k}^2}, \alpha_k^{max} \right\} \, ,
		\end{equation}
		with $L$  the Lipschitz constant of $\nabla f$ (see, e.g., \cite{bomze2020active,pedregosa2020linearly}).
	\end{itemize}
	
	The Lipschitz constant dependent step size can be seen as the minim\imbb{izer} of the quadratic model \imbb{$m_k(\cdot; L)$ overestimating $f$ along the line $\x_k+\alpha\, \dd_k$:}
	\begin{equation} \label{eq:mg}
		m_k( \alpha ; L) = f(\x_k) + \alpha\, \ps{\nabla f(\x_k)}{\dd_k} + \frac{L\alpha^2}2\,  \n{\dd_k}^2 \geq f(\x_k + \alpha\, \dd_k)\, , 
	\end{equation}
	where the inequality follows by the standard Descent Lemma \cite[Proposition 6.1.2]{bertsekas2015convex}.  
	
\imbb{In case $L$ is unknown, it is even  possible  to approximate $L$} using a backtracking line search (see, e.g., \cite{kerdreux2020affine,pedregosa2020linearly}). 


We now report a lower bound for the improvement on the objective obtained with the stepsize \eqref{eq:L}, often used in the convergence analysis.
	\begin{lemma} \label{l:am1}
		If $\alpha_k$ is given by \eqref{eq:L} and $\alpha_k < \alpha_k^{\max}$ 
		then 
		\begin{equation} \label{eq:cr}
			f(\x_{k + 1}) \leq f(\x_k) - \frac{1}{2L}(\ps{\nabla f(\x_k)}{\widehat{\dd}_k})^2  \, .
		\end{equation}
	\end{lemma}
	\begin{proof}
		We have
		\begin{equation}
			\begin{array}{rcl}
				 f(\x_k + \alpha_k \, \dd_k) &\leq &f(\x_k) + \alpha_k \ps{\nabla f(\x_k)}{\dd_k} + \frac{L\alpha_k^2}{2}\, \n{\dd_k}^2 \\[0.3em]
				&= &f(\x_k) - \frac{(\ps{\nabla f(\x_k)}{\dd_k})^2}{2L\n{\dd_k}^2} =   f(\x_k) - \frac{1}{2L}(\ps{\nabla f(\x_k)}{\widehat{\dd_k}})^2 \, ,
			\end{array}
		\end{equation}
		where we used the standard Descent Lemma in the inequality. \qed
	\end{proof}
	\section{Properties of the FW method and its variants} \label{s:cFW}
	\subsection{The FW gap}
	A key parameter often used as a measure of convergence is the FW gap 
	\begin{equation}\label{gapfun}
		G(\x) = \max_{\s \in \OM} - \ps{\nabla f(\x)}{(\s - \x)} 
		\, ,
	\end{equation}
	which is always nonnegative and equal to $0$ only in first order stationary points. This gap  is, by definition, readily available during the algorithm.  	
If $f$ is convex, using that $\nabla f(\x)$ is a subgradient we obtain
	\begin{equation}
		G(\x) \geq -  \ps{\nabla f(\x)}{(\x^* - \x)} \geq f(\x) - f^*\, ,
	\end{equation}
	so that $G(\x)$ is an upper bound on the optimality gap at $\x$. Furthermore, $G(\x)$ is a special case of the Fenchel duality gap \cite{lacoste2013block}.
	
	If $\OM=\Delta_{n - 1}$ is the simplex, then $G$ is related to the Wolfe dual \imbb{as defined in \cite{clarkson2010coresets}. 
	Indeed, {this variant of} Wolfe's dual reads}
	\begin{equation}\label{p:wd}
		\begin{aligned}
			\max \ & f(\x) + \lambda (\e^{\intercal}\x-1)- \ps{\uu}{\x} 	\\		
	\tx{s.t.}\;\;		& \nabla_i f(\x) - u_i +  {\lambda} = 0 \, ,\quad i \in \irg 1n \, , \\
			& (\x, \uu, {\lambda}) \in \R^n \times \R^n_+ \times \R \ 
		\end{aligned}
	\end{equation}
	and 
	for a fixed $\x\in \R^n$, the optimal values of $(\uu,  {\lambda})$ are 
	$$ {\lambda}_\x = - \min_{j} \nabla_j f(\x)\, , \, \ \ u_i(\x):=   \nabla_i f(\x) - \min_j  \nabla_j f(\x) \ge 0 \, .$$ 
Performing maximization in problem~\eqref{p:wd} iteratively, first for $(\uu,\lambda)$ and then for $\x$, this implies that \eqref{p:wd} is equivalent to
	\begin{equation}\begin{array}{rcl}
&&\max_{\x \in \R^n}\left [ f(\x) + \lambda_\x (\e\T\x-1) - \uu(\x)\T\x\right ]\\[0.3em] & = & \max_{\x \in \R^n}\left [ f(\x)- \max_j (\e_j-\x)\T\nabla f(\x) \right ] 
=	\max_{\x \in \R^n}\left [ f(\x) - G(\x)\right ] \, .\end{array}
	\end{equation}
	Furthermore, since Slater's condition is satisfied, strong duality holds by Slater’s theorem \cite{boyd2004convex}, \imbb{resulting in} $ G(\x^*) = 0$ for every solution $\x^*$ of the primal problem. 
	
	The FW gap is related to several other measures of convergence (see e.g. \cite[Section 7.5.1]{lan2020first}). First, consider the projected gradient
	\begin{equation}
		\widetilde{\g}_k := \pi_\OM(\x_k - \nabla f(\x_k)) - \x_k \, .
	\end{equation}
	with $\pi_{B}$ the projection on a convex and closed subset $B\subseteq\R^n$. We have $\n{\widetilde{\g}_k} = 0$ if and only if $\x_k$ is stationary, with
	\begin{equation}
		\begin{array}{rcl}
			\n{\widetilde{\g}_k}^2 & = & \ps{\widetilde{\g}_k}{\widetilde{\g}_k} \,\leq\, \ps{\widetilde{\g}_k}{[(\x_k - \nabla f(\x_k) ) - \pi_\OM(\x_k - \nabla f(\x_k))]} + \ps{\widetilde{\g}_k}{\widetilde{\g}_k} \\[0.4em]
			&= &   -\ps{\widetilde{\g}_k}{\nabla f(\x_k)} \, = \, -\ps{(\pi_\OM(\x_k - \nabla f(\x_k)) - \x_k) }{\nabla f(\x_k)} \\[0.4em]
			&\leq & \max\limits_{\y \in \OM} - \ps{(\y - \x_k)}{\nabla f(\x_k)} \,=\, G(\x_k) \, ,	
		\end{array}
	\end{equation}
	where we used $\ps{[\y - \pi_{\OM}(\x)]}{[\x - \pi_{\OM}(\x)]} \leq 0$ in the first inequality, with $\x = \x_k - \nabla f(\x_k)$ and $\y = \x_k$. \\
	Let now $N_{\OM}(x)$ denote the normal cone to $\OM$ at a point $\x \in \OM$:
	\begin{equation}
		N_{\OM}(\x) := \{\rr \in \R^n : \ps{\rr}{(\y - \x)} \leq 0 \; \mbox{ for all } \y \in \OM \} \, .
	\end{equation}
	First-order stationarity conditions are equivalent to $ - \nabla f(\x) \in N_{\OM}(\x)$, or
	$$\dist(N_{\OM}(\x), - \nabla f(\x)) = \n{- \nabla f(\x) - \pi_{N_{\OM}(\x)}(-\nabla f(\x)) } = 0 \, .$$
	The FW gap provides a lower bound on the distance from the normal cone $\dist(N_{\OM}(\x), - \nabla f(\x))$, \imbb{inflated by the diameter $D>0$ of $\OM$, as follows}:
	\begin{equation}
		\begin{array}{rcl}
			G(\x_k) &= & -\ps{(\s_k - \x_k)}{\nabla f(\x_k)} \\[0.4em]
		&= & \ps{(\s_k - \x_k)}{[\pi_{N_{\OM}(\x_k)}(-\nabla f(\x_k)) - (\pi_{N_{\OM}(\x_k)}(-\nabla f(\x_k)) + \nabla f(\x_k))]} \\[0.4em] 
			&\leq & \n{\s_k - \x_k}\, \n{\pi_{N_{\OM}(\x_k)}(-\nabla f(\x_k)) + \nabla f(\x_k)} \\[0.4em]
			&\leq & D\, \dist(N_{\OM}(\x_k), -\nabla f(\x_k)) \, ,		
		\end{array}
	\end{equation}
	where in the first inequality we used $\ps{(\s_k - \x_k)}{[\pi_{N_{\OM}(\x_k)}(-\nabla f(\x_k))]} \leq 0$ together with the Cauchy-Schwarz inequality, and $\n{\s_k - \x_k} \leq D$ in the second.
	
	\subsection{${\mathcal O}(1/k)$ rate for convex objectives}
	
If $f$ is non-convex, it is possible to prove a ${\mathcal O}(1/\sqrt{k})$ rate for $\min_{i \in [1:k]} G(x_i)$ (see, e.g., \cite{lacoste2016convergence}). On the other hand, if $f$ is convex, we have an ${\mathcal O}(1/k)$ rate on the optimality gap (see, e.g., \cite{frank1956algorithm,levitin1966constrained}) for all the stepsizes discussed in Section~\ref{s:stp}. Here we include a proof for the Lipschitz constant dependent stepsize $\alpha_k$ given by \eqref{eq:L}.
	\begin{theorem} \label{t:1/k}
		If $f$ is convex and the stepsize is given by \eqref{eq:L}, then 
for every $k \geq 1$		
\begin{equation}\label{thm1cond}
			f(\x_k) - f^* \leq \frac{2LD^2}{k + 2}  \, .
		\end{equation}
	\end{theorem}
	Before proving the theorem we prove a lemma concerning the decrease of the objective in the case of a full FW step, that is a step with $\dd_k = \dd_k^{FW}$ and with $\alpha_k$ equal to $1$, the maximal feasible stepsize. 
	\begin{lemma}\label{l:g2}
		If $\alpha_k = 1$ and $\dd_k = \dd_k^{FW}$ then 
		\begin{equation} \label{eq:lg2}
			f(\x_{k + 1}) - f^* \leq \frac 12 \, \min \left\{ L\n{\dd_k}^2 ,   f(\x_k) - f^*  \right\} \, . 		
		\end{equation}	
	\end{lemma}
	\begin{proof}
		If $\alpha_k = 1 = \alpha_k^{\max}$ then \imbb{by Definitions~\eqref{fwdk} and \eqref{gapfun}}
		\begin{equation} \label{eq:l11}
			G(\x_k) = -\ps{\nabla f(\x_k)}{\dd_k} \geq L \n{\dd_k}^2 \, ,
		\end{equation}
	\imbb{the last inequality following by Definition~\eqref{eq:L} and the assumption that $\alpha_k = 1$.	By the standard Descent Lemma it also follows}
		\begin{equation} \label{eq:l1sd}
			f(\x_{k + 1}) - f^* = f(\x_k + \dd_k) - f^* \leq f(\x_k) - f^* + \ps{\nabla f(\x_k)}{\dd_k} + \frac{L}{2}\,\n{\dd_k}^2 \, .
		\end{equation}
		\imbb{Considering the definition of $\dd_k$  and convexity of $f$, we get 
		$$f(\x_k) - f^* + \ps{\nabla f(\x_k)}{\dd_k} \le 
		f(\x_k) - f^* + \ps{\nabla f(\x_k)}{(\x^*-\x_k)} \le
		0\, ,$$
		so that~\eqref{eq:l1sd} entails $	f(\x_{k + 1}) - f^* \le \frac{L}{2}\,\n{\dd_k}^2 
		$.}
		To conclude, it suffices to apply to the RHS \imbb{of~\eqref{eq:l1sd}} the inequality
		\begin{equation}
				f(\x_k) - f^* + \ps{\nabla f(\x_k)}{\dd_k} + {\textstyle{\frac{L}{2}}}\,\n{\dd_k}^2 \leq 
				f(\x_k) - f^* -{\textstyle{\frac{1}{2}}}\,G(\x_k) 
				\leq 
				{\textstyle{\frac{f(\x_k) - f^*}{2}}}\,	
		\end{equation}
		where we used \eqref{eq:l11} in the first inequality and $G(\x_k) \geq f(\x_k) - f^*$ in the second. \qed
	\end{proof}

	We can now proceed with the proof of the main result. 
	\begin{proof}[Theorem \ref{t:1/k}]
		For  $k = 0$ and $\alpha_0 = 1$ then by Lemma \ref{l:g2}
		\begin{equation}
			f(\x_1) - f^* \leq \frac{L\n{\dd_0}^2}{2} \leq \frac{ L D^2}{2} \, . 
		\end{equation}
		If $\alpha_0 < 1$ then 
		\begin{equation}
			f(\x_0) - f^* \leq G(\x_0) < L\n{\dd_0}^2 \leq LD^2 \, .
		\end{equation}
		Therefore in both cases \eqref{eq:cr} holds for $k = 0$. \\
		Reasoning by induction, if~\eqref{thm1cond} holds for $k$ \imbb{with} $\alpha_k = 1$, then the claim is clear by~\eqref{eq:lg2}. On the other hand, if $\alpha_k <\alpha_k^{\max}= 1  $ then by Lemma \ref{l:am1}, we have 
		\begin{equation}
			\begin{array}{rcl}
			 f(\x_{k + 1}) - f^* &\leq  &f(\x_k) - f^* - \frac{1}{2L}\, (\ps{\nabla f(x_k)}{\widehat{\dd}_k})^2\\[0.4em]
			 &\leq &f(\x_k) - f^* - \frac{(\ps{\nabla f(\x_k)}{\dd_k})^2}{2LD^2} \\[0.4em] &\leq &f(\x_k) - f^* - \frac{(f(\x_k) - f^*)^2}{2LD^2}\\[0.4em] 
			 &= &(f(\x_k) - f^*)(1 - \frac{f(\x_k) - f^*}{2LD^2}) \, \leq \, \frac{2LD^2}{k + 3} \, ,
			\end{array}	
		\end{equation}
		where we used $\n{\dd_k} \leq D$ in the second inequality, $\ps{\nabla f(\x_k)}{\dd_k} = G(\x_k) \geq f(\x_k) - f^*$ in the third inequality; and the last inequality follows by 
		induction hypothesis. \qed 
	\end{proof}
	As can be easily seen from  above argument, the convergence rate of ${\mathcal O}(1/k)$ is true also in more abstract normed spaces than $\R^n$, e.g. when $\OM$ is a convex and weakly compact subset of a Banach space (see, e.g., \cite{demianov1970approximate,dunn1978conditional}).  A generalization for some unbounded sets is given in \cite{ferreira2021frank}. The bound is tight due to a zigzagging behaviour of the method near solutions on the boundary, leading to a rate of $\Omega(1/k^{1 + \delta})$ for every $\delta > 0$ (see \cite{canon1968tight} for further details), when the objective is a strictly convex quadratic function and the domain is a polytope. \\
	Also the minimum FW gap $\min_{i \in [0: k]} G(\x_i) $ converges at a rate of ${\mathcal O}(1/k)$ (see \cite{jaggi2013revisiting,freund2016new}). In \cite{freund2016new}, a broad class of stepsizes is examined, including $\alpha_k= \frac 1{k + 1}$ and $\alpha_k = \bar{\alpha}$ constant. For these stepsizes a convergence rate of ${\mathcal O}\left(\frac{\ln(k)}{k}\right)$ is proved.
	
	\subsection{Variants} \label{s:var}
	Active set FW variants mostly aim to improve over the ${\mathcal O}(1/k)$ rate and also ensure support identification in finite time. They generate a sequence of active sets $\{A_k\}$, such that $\x_k \in \conv(A_k)$, and define alternative directions making use of these active sets. 
	
	For the {\em pairwise FW (PFW)} and the {\em away step 
	FW (AFW)} (see  \cite{clarkson2010coresets,lacoste2015global}) we have that $A_k$ must always be a subset of $S_k$, with $\x_k$ a convex combination of the elements in $A_k$. The away vertex $\vv_k$ is then defined by 
	\begin{equation}
		\vv_k \in \argmax_{\y \in A_k} \ps{\nabla f(\x_k)}{\y} \, .
	\end{equation}
	The AFW direction, introduced in \cite{wolfe1970convergence}, is hence given by
	\begin{equation}
		\begin{array}{ll}
			\dd^{AS}_k &= \x_k - \vv_k \\
			\dd_k   &\in \argmax\{\ps{-\nabla f(\x_k)}{\dd} : \dd \in \{\dd_k^{AS}, \dd_k^{FW}\} \} \, ,
		\end{array}
	\end{equation}
	while the PFW direction, as defined in \cite{lacoste2015global} and inspired by the early work \cite{mitchell1974finding}, is
	\begin{equation}
		\dd^{PFW}_k =\dd_k^{FW}+\dd^{AS}_k= \s_k - \vv_k \, ,
	\end{equation}
	\imbb{with $\s_k$ defined in~\eqref{fwdk}.}
	
	The {\em FW method with in-face directions (FDFW)} (see \cite{freund2017extended,guelat1986some}), also known as Decomposition invariant Conditional Gradient (DiCG) when applied to polytopes~\cite{bashiri2017decomposition}, is defined exactly as the AFW, but with the minimal face $\mathcal{F}(\x_k)$ of $\OM$ containing $\x_k$ as the active set.  
	The {\em extended FW (EFW)} was introduced in \cite{holloway1974extension} and is also known as simplicial decomposition \cite{von1977simplicial}. At every iteration the method minimizes the objective in the current active set $A_{k + 1}$
	\begin{equation}
		\x_{k + 1} \in \argmin_{\y \in \conv(A_{k + 1})} f(\y)\,  ,
	\end{equation}
	where $A_{k + 1} \subseteq A_k \cup \{s_k\}$ (see, e.g., \cite{clarkson2010coresets}, Algorithm 4.2). A more general version of the EFW, only approximately minimizing on the current active set, was introduced in \cite{lacoste2015global} under the name of fully corrective FW.  In Table \ref{tab:FWvariants}, we report the main features of the classic FW and of the variants under analysis.\\
	
	\begin{table}[h]
		\begin{center}
			\def\arraystretch{1.5}
			\begin{tabular}{|l|l|c|}
				\hline
				\multicolumn{1}{|c|}{Variant} & \multicolumn{1}{|c|}{Direction}   &  \multicolumn{1}{|c|}{Active set}     \\ \hline
				FW      & $\dd_k = \dd_k^{FW} = \s_k - \x_k, \quad \s_k \in \argmax\{\ps{\nabla f(\x_k)}{\x} : \x \in \OM \}  $ & -                                       \\ \hline
				AFW     &  $	\dd_k   \in \argmax\{\ps{-\nabla f(\x_k)}{\dd} : \dd \in \{\x_k - \vv_k, \dd_k^{FW}\}, \ \vv_k \in A_k \}$ \        & $A_{k + 1}\subseteq A_k \cup \{\s_{k}\}$   \\ \hline
				PFW     &  $	\dd_k = \s_k - \vv_k, \quad
				\vv_k        \in \argmax\{\ps{\nabla f(\x_k)}{\vv_k} : \vv_k \in A_k \} $	            &             $A_{k + 1}\subseteq A_k \cup \{\s_{k}\}$                             \\ \hline
				EFW    &     		$		\dd_k = \y_k - \x_k, \quad 	\y_k       \in \argmin\{f(\y): \y \in \conv(A_k) \}           $
				&    $A_{k + 1}\subseteq A_k \cup \{\s_{k}\}$                  \\ \hline
				FDFW    &   $	\dd_k   \in \argmax\{\ps{-\nabla f(\x_k)}{\dd} : \dd \in \{\x_k - \vv_k, \dd_k^{FW}\}, \ \vv_k \in A_k \} $           & $A_k = \F(\x_k)$   \\  \hline                   
			\end{tabular}
			\caption{FW method and variants covered in this review.}\label{tab:FWvariants}
		\end{center}
	\end{table}	
	
	\subsection{Sparse approximation properties}
	As discussed in the previous section, for the classic FW method and the AFW, PFW, EFW variants $\x_k$ can always be written as a convex combination of elements in $A_k \subset S_k$, with $|A_k| \leq k + 1$. Even for the FDFW we still have the weaker property that $\x_k$ must be an affine combination of elements in $A_k \subset A$ with $ |A_k| \leq k + 1$. It turns out that the convergence rate of methods with this property is $\Omega(\frac{1}{k})$ in high dimension. More precisely, if $\OM = \conv(A)$ with $A$ compact, the ${\mathcal O}(1/k)$ rate of the classic FW method is worst case optimal given the sparsity constraint 
	\begin{equation} \label{eq:sparse}
		x_k \in \aff(A_k)  \ \text{with }A_k \subset A, \ |A_k| \leq k + 1 \, .
	\end{equation}
	An example where the ${\mathcal O}(1/k)$ rate is tight was presented in \cite{jaggi2013revisiting}. Let $\OM = \Delta_{n - 1}$ and $f(\x) = \n{\x - \frac 1n\, \e }^2$. Clearly, $f^* = 0$ with $\x^* = \frac 1n\, \e$. Then it is easy to see that $\min \{f(\x) - f^* : {\|\x\|}_0 \leq k + 1 \} \geq \frac{1}{k + 1} - \frac{1}{n}$ for every $k \in \mathbb{N}$, so that in particular under \eqref{eq:sparse} with $A_k = \{e_i  : i \in \irg 1n\}$, the rate of any FW variant must be $\Omega(\frac{1}{k})$. 
	
	\subsection{Affine invariance}
	
	The FW method and the AFW, PFW, EFW are affine invariant \cite{jaggi2013revisiting}. More precisely, let $\Pb$ be a linear transformation, $\hat{f}$ be such that $\hat{f}(\Pb\x) = f(\x)$ and $\hat{\OM} = \Pb(\OM)$. Then for every sequence $\{\x_k\}$ generated by the methods applied to $(f, \OM)$, the sequence $\{\y_k\} := \{\Pb\x_k\}$ can be generated by the FW method with the same stepsizes applied to $(\hat{f}, \hat{\OM})$. As a corollary, considering the special case where $\Pb$ is the matrix collecting the elements of $A$ as columns, one can prove results on $\OM = \Delta_{|A| - 1}$ and generalize them to $\hat{\OM}:= \conv(A)$ by affine invariance. 
	
	An affine invariant convergence rate bound for convex objectives can be given using the curvature constant
	\begin{equation}
		\kappa_{f, \OM} := 
		\sup\left\{2{\textstyle{ \frac{f( \alpha \y+(1-\alpha)\x ) - f(\x)	- \alpha \ps{\nabla f(\x)}{(\y-\x)}}{\alpha^2}}}: \{\x,\y\}\subset \OM, \,  
		\alpha \in (0, 1]
		 \right\} \, .
		%
	\end{equation}
	It is easy to prove that $\kappa_{f, \OM} \leq LD^2$ if $D$ is the diameter of $\OM$. In the special case where $\OM = \Delta_{n - 1}$ and $f(\x) = \x^\intercal \tilde{\Ab}^\intercal \tilde{\Ab} \x + \bb\T\x$, then $\kappa_{f, \OM} \leq \diam(\Ab\Delta_{n - 1})^2$ for \imbb{$\Ab\T = [\tilde{\Ab}\T, \bb]$}; see \cite{clarkson2010coresets}. 
	
	When the method uses the stepsize sequence \eqref{alpha2k2}, it is possible to give the following affine invariant convergence rate bounds (see \cite{freund2016new}):
	\begin{equation}
		\begin{aligned}
			f(\x_k) - f^* & \leq \frac{2\kappa_{f, \OM}}{k + 4} \, , \\
			\min_{i \in ]0:k]} G(\x_i) & \leq \frac{9\kappa_{f, \OM}}{2k} \, , 
		\end{aligned}
	\end{equation}
	thus in particular slightly improving the rate we gave in Theorem \ref{t:1/k} since we have that $\kappa_{f, \OM} \leq LD^2$. \\
	
	\subsection{Support identification for the AFW}
	It is a classic result that the AFW under some strict complementarity conditions and for strongly convex objectives identifies in finite time the face containing the solution \cite{guelat1986some}. Here we report some explicit bounds for this property proved in \cite{bomze2020active}. We first assume that $\OM = \Delta_{n - 1}$, and introduce the multiplier functions
	\begin{equation}
		\lambda_i(\x) = \ps{\nabla f(\x)}{(\e_i - \x)}
	\end{equation} 
	for $i \in \irg 1n$. Let $\x^*$ be a stationary point for $f$, with the objective $f$ not necessarily convex. It is easy to check that $\{\lambda_i(\x^*)\}_{i \in [1:n]}$ coincide with the Lagrangian multipliers.
	Furthermore, by complementarity conditions we have $x^*_i \lambda_i(\x^*) = 0$ for every $i \in  \irg 1n$. \imbb{It follows that the set
	 $$I(\x^*) := \{i \in  \irg 1n : \lambda_{i}(\x^*) = 0\}$$
	 contains the support of $\x^*$,
	 $$\supp(\x^*) :=\{ i\in \irg 1n: x_i^*>0\}\, .$$}
	The next lemma uses $\lambda_i$, and the Lipschitz constant $L$ of $\nabla f$, to \imbb{give a lower bound of the so-called {\em active set radius} $r_*$, defining a neighborhood of $\x^*$. Starting the algorithm in this neighbourhood, the active set (the minimal face of $\OM$ containing $\x^*$) is identified in a limited number of iterations.}  
	\begin{lemma} \label{ascmain}
		Let $\x^*$ be a stationary point for $f$ on the boundary of $\Delta_{n - 1}$, $\delta_{\min} = \min_{i: \lambda_{i}(\x^*) > 0} \lambda_i(\x^*)$ and
		\begin{equation}
			r_* = \frac{\delta_{\min}}{\delta_{\min} + 2L} \, .
		\end{equation}
		Let. 
		Assume that for every $k$ for which $\dd_k = \dd_k^{\mathcal{A}}$ holds, the step size $\alpha_k$ is \imbb{not smaller than} the stepsize given by \eqref{eq:L}, $\alpha_k(L) \le \alpha_k$.\\ 
		If ${\| \x_k - \x^*\|}_1 < r_*$, then for some 
		$$j \leq \min\{ n - |I(\x^*)|, |\supp(\x_k)| - 1\}$$ we have $\supp(\x_{k + j}) \subseteq I(\x^*)$ and $\n{\x_{k + j} - \x^*}_1 < r_*$. 
	\end{lemma}
	\begin{proof}
		Follows from \cite[Theorem 3.3]{bomze2020active}, since under the assumptions the AFW sets one variable in $\supp(x_k)\sm I(\x^*)$ to zero at every step without increasing the $1$-norm distance from $\x^*$. \qed
	\end{proof} 
	The above lemma does not require convexity and was applied in \cite{bomze2020active} to derive active set identification bounds in several convex and non-convex settings. Here we focus on the case where the domain $\OM = \conv(A)$ with $|A| < + \infty$ is a generic polytope, and \imbb{ where $f$ is $\mu$-strongly convex for some $\mu>0$, i.e.
	\begin{equation}\label{mustr}
	f(\y) \geq f(\x) + \ps{\nabla f(\x)}{(\y - \x)} + \frac{\mu}{2} \n{\x - \y}^2\quad\mbox{for all } \{\x, \y\} \subset C\, .
\end{equation}}
	 Let $E_{\OM}(\x^*)$ be the face of $\OM$ exposed by $\nabla f(x^*)$:
	\begin{equation}
		E_{\OM}(\x^*) := \argmin_{\x \in \OM} \ps{\nabla f(\x^*)}{\x} \, ,
	\end{equation}
		Let then $\theta_{A}$ be the Hoffman constant (see \cite{beck2017linearly}) related to \imbb{$[\bar{\Ab}\T, \Ib_n,  \e, -\e]\T$}, with $\bar{\Ab}$ the matrix having as columns the elements in $A$. Finally, \imbb{consider the function  $ {f}_A(\y) := f(\bar{\Ab}\y)$ on $ \Delta_{|A| - 1}$, and} let $L_{A}$ be the Lipschitz constant of $\imbb{\nabla} {f_A}$ as well as 
		$$\imbb{
		\delta_{\min} := \min_{\av \in A \sm E_{\OM}(\x^*)}\ps{\nabla f(\x^*)}{(\av - \x^*)}\quad\mbox{and}\quad  r_*(\x^*) := \frac{\delta_{\min}}{\delta_{\min} + 2L_A}\, .}
		$$
		Using \imbb{linearity of} AFW convergence for strongly convex objectives (see Section~\ref{ss:lc}), we have the following result:
	\begin{theorem} 	\label{t:t1}	
	 The sequence $\{\x_k\}$ generated by the AFW with $\x_0 \in A$ \imbb{enters} $E_{\OM}(\x^*)$ for
		\begin{equation}
			k \geq \max\left \{ 2\frac{\ln(h_0) - \ln(\mu_A r_*(\x^*)^2/2)}{\ln(1/q)}, 0\right \} \, ,
		\end{equation}
		where $\mu_A = \frac{\mu}{n\theta_{A}^2}$ and $q\in(0,1)$ is the constant related to the linear convergence rate of the AFW, i.e. $h_k\le q^k h_0$ for all $k$.
			\end{theorem}
	\begin{proof}[sketch]
		We present an argument in the case $\OM = \Delta_{n - 1}$, $A = \{e_i\}_{i \in [1:n]}$ which can be easily extended by affine invariance to the general case (see \cite{bomze2020active} for details). In this case $\theta_{A} \geq 1$ and we can define $\imbb{\bar{\mu} := {\mu}/n }\geq \mu_{A} $. \\
		To start with, the number of steps needed 
		to reach the condition 
		\begin{equation} \label{hkleq}
			h_k \leq \frac{\mu}{2n}r_*(\x^*)^2 = \frac{\bar{\mu}}{2} r_*(\x^*)^2
		\end{equation}
		is at most 
		$$ \bar{k} = \max \left\{ \left\lceil\frac{\ln(h_0) - \ln(\bar{\mu} r_*(x^*)^2/2)}{\ln(1/q)}\right\rceil, 0\right\}\ . $$	
				\imbb{Now we combine $n\n{\cdot} \geq {\n{\cdot}}_1$ with strong convexity and relation~\eqref{hkleq}} to obtain ${\n{\x_k - \x^*}}_1 \leq r_*(\x^*)$, hence in particular ${\n{\x_k - \x^*}}_1 \leq r_*(\x^*)$ for every $k \geq \bar{k}$. Since $\x_0$ is a vertex of the simplex, and at every step at most one coordinate is added to the support of the current iterate, $|\supp(\x_{\bar{k}})| \leq \bar{k} + 1$. The claim follows by applying Lemma~\ref{ascmain}. \qed 	\end{proof}
	Additional bounds under a quadratic growth condition weaker than strong convexity and strict complementarity are reported in \cite{garber2020revisiting}. 
	
	Convergence and finite time identification for the PFW and the AFW are proved in \cite{bomze2019first} for a specific class of non-convex minimization problems over the \imbb{standard} simplex, under the additional assumption that  the sequence generated has a finite set of limit points. In another line of work, active set identification strategies combined with FW variants have been proposed in \cite{cristofari2020active} and \cite{sun2020safe}.
	
	\subsection{Inexact linear oracle}
	
	In many real-world applications, linear subproblems can only be solved approximately. This is the reason why the convergence of FW variants is often analyzed under some error term for the linear minimization oracle (see, e.g., \cite{braun2019blended,braun2017lazifying,freund2016new,jaggi2013revisiting,konnov2018simplified}). A common assumption, relaxing the FW vertex exact minimization property, is to have access to a point (usually a vertex) $\tilde{\s}_k$ such that
	\begin{equation} \label{eq:fwer}
		\ps{\nabla f(\x_k)}{(\tilde{\s}_k-\x_k)} \leq \min_{\s \in \OM} \ps{\nabla f(\x_k)}{(\s-\x_k)} + \delta_k \, ,
	\end{equation} 
	for a sequence $\{\delta_k\}$ of non negative approximation errors. \\
	If the sequence $\{\delta_k\}$ is constant and equal to some $\delta > 0$, then trivially the lowest possible approximation error achieved by the FW method is $\delta$. At the same time, \cite[Theorem 5.1]{freund2016new} implies a rate of ${\mathcal O}(\frac{1}{k} + \delta)$ if the stepsize $\alpha_k= \frac{2}{k + 2}$ is used. 
	
	The ${\mathcal O}(1/k)$ rate can be instead retrieved by assuming that $\{\delta_k\}$ converges to $0$ quickly enough, and in particular if 
	\begin{equation} \label{eq:eras}
		\delta_k  = \frac{\delta \kappa_{f, C}}{k + 2}
	\end{equation}
	for a constant $\delta > 0 $. Under \eqref{eq:eras}, in \cite{jaggi2013revisiting} a convergence rate of 
	\begin{equation}
		f(x_k) - f^* \leq \frac{2\kappa_{f, \OM}}{k + 2}(1 + \delta)
	\end{equation}
	was proved for the FW method with $\alpha_k$ given by exact line search or equal to $\frac 2{k + 2}$, as well as for the EFW. 
	
	A linearly convergent variant making use of an approximated linear oracle recycling previous solutions to the linear minimization subproblem is studied in \cite{braun2019blended}. In \cite{freund2016new,hogan1971convergence}, the analysis of the classic FW method is extended to the case of inexact gradient information. In particular in \cite{freund2016new}, assuming the availability of the $(\delta, L)$ oracle introduced in \cite{devolder2014first}, a convergence rate of ${\mathcal O}(1/k + \delta k)$ is proved.
	\section{Improved rates for strongly convex objectives} \label{s:ir}
	\begin{table}[]
		\begin{center}
			\def\arraystretch{1.5}
			\begin{tabular}{|c|c|c|c|c|}
				\hline
				Method   & Objective & Domain          & Assumptions                                                                        & Rate            \\ \hline
				FW       & NC        & Generic         & -                                                                                  & ${\mathcal O}(1/\sqrt{k})$ \\ \hline
				FW       & C         & Generic         & -                                                                                  & ${\mathcal O}(1/k)$        \\ \hline
				FW       & SC        & Generic        & $\x^* \in \ri(\Omega)$                                                           & Linear          \\ \hline
				Variants & SC        & Polytope        & -                                                                                  & Linear          \\ \hline
				FW       & SC        & Strongly convex & -                                                                                  & ${\mathcal O}(1/k^2)$      \\ \hline
				FW       & SC        & Strongly convex & $\min \n{\nabla f(\x) } > 0 $ & Linear          \\ \hline
			\end{tabular}
			\caption{Known convergence rates for the FW method and the variants covered in this review. NC, C and SC stand for non-convex, convex and strongly convex respectively. } \label{tab:2}
		\end{center}
	\end{table}	
	\subsection{Linear convergence under an angle condition} \label{ss:lc}
In the rest of this section we assume that $f$ is $\mu$-strongly convex~\eqref{mustr}.
 We also assume that the stepsize is given by exact linesearch or by \eqref{eq:L}. 
 
Under this assumption, an asymptotic linear convergence rate for the FDFW on polytopes was given in the early work \cite{guelat1986some}. Furthermore, in \cite{garber2016linearly} a linearly convergent variant was proposed, making use however of an additional local linear minimization oracle. 

Recent works obtain linear convergence rates by proving the angle condition 
\begin{equation} \label{eq:angle}
\imbb{	\ps{- \nabla f(\x_k)}{\widehat{\dd}_k} \geq  \frac{\tau}{{\n{\x_k - \x^*}}} \, \ps{\nabla f(\x_k)}{ ( \x_k - \x^*)}}
\end{equation}
 for some \imbb{$\tau >0$ and some} $\x^* \in \argmin_{\x \in C} f(\x)$. As we shall see in the next lemma, under \eqref{eq:angle} it is not difficult to prove linear convergence rates in the number of \textit{good steps}. These are FW steps with $\alpha_k = 1$ and steps in any descent direction with $\alpha_k < 1$. 
\begin{lemma} \label{l:ani}
	If the step $k$ is a good step and \eqref{eq:angle} holds, then
	\begin{equation}
\imbb{	h_{k+1} \leq \max \left\{{\textstyle{ \frac{1}{2}}, 1 -  \frac{\tau^2\mu}{L}} \right\} h_k 
	\, .}
	\end{equation}
\end{lemma}
\begin{proof}
	If the step $k$ is a full FW step then \imbb{Lemma \ref{l:g2} entails
	$h_{k+1}\le \frac 12\, h_k$.}
	In the remaining case, first observe that by strong convexity
	\begin{equation}
		\begin{array}{rcl}
			 f^* &= &f(\x^*) \geq f(\x_k) + \ps{\nabla f(\x_k)}{(\x^* - \x_k)} + \frac{\mu}{2}\n{\x_k - \x^*}^2 \\[0.5em]
			&\geq &\min\limits_{\alpha \in \mathbb{R}}\left [ f(\x_k) + \alpha \ps{\nabla f(\x_k)}{(\x^* - \x_k)} + \frac{\alpha^2\mu}{2}\n{\x_k - \x^*}^2 \right] \\[0.5em] 
			&= &f(\x_k) - \frac{1}{2\mu {\n{\x_k - \x^*}}^2}	\left[ \ps{\nabla f(\x_k)}{( \x_k - \x^*)}\right]^2 \, ,
		\end{array}
	\end{equation}
\imbb{which means 
\begin{equation}\label{eq:scb}h_k \le \frac 1{2\mu {\n{\x_k - \x^*}}^2}\left[ \ps{\nabla f(\x_k)}{(\x_k - \x^*)}\right]^2 \, .\end{equation}	
We can then proceed using} the bound~\eqref{eq:cr} from Lemma~\ref{l:am1} in the following way:
	\begin{equation}
		\begin{array}{rcl}
	h_{k+1} &=		&f(\x_{k + 1}) - f^* \leq f(\x_k) - f^* - \frac{1}{2L}\left [\ps{\nabla f(\x_k)}{\widehat{\dd}_k}\right]^2 \\[0.5em] 
			&\leq & h_k - \frac{\tau^2}{2L{\n{\x_k - \x^*}}^2}  \left[ \ps{\nabla f(\x_k)}{ (  \x_k - \x^* )} \right]^2 \\[0.5em]
			&\leq &h_k \left(1 - \frac{ \tau^2\mu}{L}\right) \, ,		
		\end{array}	
	\end{equation}
	where we used \eqref{eq:angle} in the \imbb{second} inequality and \eqref{eq:scb} in the \imbb{third} one. 	\qed
\end{proof}
As a corollary, under \eqref{eq:angle} we have the rate
\begin{equation}
	f(\x_{k}) - f^* \imbb{= h_k} \leq  \max \left\{{\textstyle{ \frac{1}{2}},  1 -  \frac{\tau^2\mu}{L}}\right\}^{\gamma(k)} \imbb{h_0}
\end{equation}
for any method with non increasing $\{f(\x_k)\}$ and following Algorithm \ref{alg:GS}, with $\gamma(k) \leq k$ an integer denoting the number of good steps until step $k$.  It turns out that for all the variants we introduced in this review we have $\gamma(k) \geq Tk$ for some constant $T > 0$.  When $\x^*$ is in the relative interior of $\OM$, the FW method satisfies \eqref{eq:angle} and we have the following result (see \cite{guelat1986some,lacoste2015global}):
\begin{theorem}
	If $\x^* \in \ri(\OM)$, then
	\begin{equation}
		f(\x_{k}) - f^* \leq \left [ 1 - \frac{\mu}{L} \left(\frac{\dist(\x^*, \partial \OM)}{D}\right)^2\right]^k (f(\x_0) - f^*) \, .
	\end{equation}	
\end{theorem}
\begin{proof}
	We can assume for simplicity $\intp(\OM) \neq \emptyset$, since otherwise we can restrict ourselves to the affine hull of $\OM$. Let $\delta=\dist(\x^*, \partial \OM)$ and $\g = -\nabla f(\x_k)$. First, by assumption we have $\x^* + \delta \widehat{\g} \in \OM$. Therefore
	\begin{equation}
		\ps{\g}{\dd_k^{FW}} \geq \ps{\g}{((\x^*+\delta \widehat{\g} ) - \x)} = \delta \ps{\g}{\widehat{\g}} + \ps{\g}{(\x^* - \x)} \geq \delta \n{\g} + f(\x) - f^* \geq \delta \n{\g} \, ,
	\end{equation}
	where we used $\x^*+\delta \widehat{\g}  \in \OM$ in the first inequality and convexity in the second. We can conclude
	\begin{equation}
		\ps{\g}{\frac{\dd_k^{FW}}{\n{\dd_k^{FW}}}} \geq \ps{\g}{\frac{\dd_k^{FW}}{D}} \geq  \frac{\delta}{D} \n{\g} \geq \frac{\delta}{D} \ps{\g}{\left(\frac{\x_k - \x^*}{\n{\x_k - \x^*}}\right)} \, .
	\end{equation}
The thesis follows by Lemma~\ref{l:ani}, noticing that for $\tau = \frac{\dist(\x^*, \partial \OM)}{D} \leq \frac{1}{2}$ we have $1 - \tau^2\frac{\mu}{L} > \frac{1}{2}$. \qed 
\end{proof}

In \cite{lacoste2015global}, the authors proved that directions generated by the AFW and the PFW on polytopes satisfy condition \eqref{eq:angle}, with $\tau = \PWidth(A)/D$ and $\PWidth(A)$,  pyramidal width of $A$. While $\PWidth(A)$ was originally defined with a rather complex minmax expression, in \cite{pena2018polytope} it was then proved 
\begin{equation}
	\PWidth(A) =	\min_{F \in \faces(C)} \dist(F, \conv(A \sm F)) \, .
\end{equation}
This quantity can be explicitly computed in a few special cases. For $A = \{0, 1\}^n$ we have $\PWidth(A) = 1/\sqrt{n}$, while for $A = \{e_i\}_{i \in [1:n]}$ (so that $\OM$ is the $n - 1$ dimensional simplex)
\begin{equation}
	\PWidth(A) = \begin{cases}
		\frac{2}{\sqrt{n}} &\tx{ if }n \tx{ is even} \\
		\frac{2}{\sqrt{n - 1/n}} &\tx{ if }n \tx{ is odd.}  
	\end{cases}
\end{equation}
Angle conditions like \eqref{eq:angle} with $\tau$ dependent on the number of vertices used to represent $x_k$ as a convex combination were given in \cite{bashiri2017decomposition} and \cite{beck2017linearly} for the FDFW and the PFW respectively. In particular, in \cite{beck2017linearly} a geometric constant $\Omega_{\OM}$ called vertex-facet distance was defined as
\begin{equation}
	\Omega_{\OM} = \min  \{\dist(\vv, H) : \vv \in V(\OM) \, , H \in {\mathcal H}(\OM), \, \vv \notin H \} \, ,
\end{equation}
with $V(\OM)$ the set of vertices of $\OM$, and ${\mathcal H}(\OM)$ the set of supporting hyperplanes of $\OM$ (containing a facet of $\OM$). Then condition \eqref{eq:angle} is satisfied for $\tau= \Omega_{\OM}/s$, with $\dd_k$ the PFW direction and $s$ the number of points used in the active set $A_k$. \\
In \cite{bashiri2017decomposition}, a geometric constant $H_s$ was defined depending on the minimum number~$s$ of vertices needed to represent the current point $x_k$, as well as on the proper\footnote{i.e., those inequalities strictly satisfied for some $\x \in \OM$.}  inequalities $\q_i\T\x \leq b_i$, $i \in \irg 1m$, appearing in a description of $\OM$. For each of these inequalities the {\em second gap} $g_i$ was defined as 
\begin{equation}
	g_i = 	\max_{\vv \in V(\OM)} \ps{\q_i}{\vv} - \secondmax_{\vv \in V(\OM)} \ps{\q_i}{\vv} \, ,\quad i\in \irg 1m\, ,
\end{equation}
with the secondmax function giving the second largest value achieved by the argument. Then $H_s$ is defined as 
\begin{equation}
	H_s := \imbb{\max {\textstyle{\left\{  \sum\limits_{j = 1}^n \left(\sum\limits_{i \in S} \frac{a_{ij}}{g_i} \right)^2 : S \in { {[1:m] }\choose s} \right \}}}}\, .
\end{equation}
The arguments used in the paper imply that \eqref{eq:angle} holds with \imbb{$\tau= \frac 2{D\sqrt{H_s}}$} if $\dd_k$ is a FDFW direction and $\x_k$ the convex combination of at most $s$ vertices. We refer the reader to \cite{pena2018polytope} and \cite{rademacher2020smoothed} for additional results on these and related constants. 

The linear convergence results for strongly convex objectives are extended to compositions of strongly convex objectives with affine transformations in \cite{beck2017linearly}, \cite{lacoste2015global}, \cite{pena2018polytope}. In \cite{gutman2020condition}, the linear convergence results for the AFW and the FW method with minimum in the interior are extended with respect to a generalized condition number $L_{f, \OM, D}/\mu_{f, \OM, D}$, with $D$ a distance function on $\OM$. 

For the AFW, the PFW and the FDFW, linear rates with no bad steps ($\gamma(k) = k$) are given in \cite{rinaldi2020avoiding} for non-convex objectives satisfying a Kurdyka-\L ojasiewicz inequality. In \cite{rinaldi2020unifying}, condition \eqref{eq:angle} was proved for the FW direction and orthographic retractions on some convex sets with smooth boundary. The work \cite{combettes2020boosting} introduces a new FW variant using a subroutine to align the descent direction with the projection on the tangent cone of the negative gradient, thus implicitly maximizing $\tau$ in \eqref{eq:angle}.  \\			
	
	\subsection{Strongly convex domains} \label{s:scd}
	When $\OM$ is strongly convex we have a ${\mathcal O}(1/k^2)$ rate (see, e.g., \cite{garber2015faster,kerdreux2021projection}) for the classic FW method. Furthermore, when $\OM$ is 
	$\beta_{\OM}$-strongly convex and $\n{\nabla f(\x)} \geq c > 0$, then we have the linear convergence rate (see \cite{demianov1970approximate,dunn1979rates,kerdreux2020affine,levitin1966constrained})
	\begin{equation}
	\imbb{h_{k+1}} \leq \max \left\{{\textstyle{\frac{1}{2}, 1 - \frac{L}{2c\beta_{\OM}} }}\right\} \imbb{h_k}
	\, .
	\end{equation}
	Finally, it is possible to interpolate between the ${\mathcal O}(1/k^2)$ rate of the strongly convex setting and the ${\mathcal O}(1/k)$ rate of the general convex one by relaxing strong convexity of the objective with H\"olderian error bounds \cite{xu2018frank} and also by relaxing strong convexity of the domain with uniform convexity \cite{kerdreux2021projection}.	
	\section{Extensions} \label{s:E}
	\subsection{Block coordinate Frank-Wolfe method}
	The block coordinate FW (BCFW) was introduced in \cite{lacoste2013block} for block product domains of the form $\OM = \OM^{(1)} \times ... \times \OM^{(m)} \subseteq \R^{n_1 + ... + n_m} $, and applied to structured SVM training. The algorithm operates by selecting a random block and performing a FW step in that block. Formally, for $\s \in \R^{m_i}$ let $\s^{(i)} \in \R^n$ be the vector with all \imbb{blocks} equal to $\oo$ except for the $i$-th \imbb{block} equal to $\s$. We can write the direction of the BCFW as
	\begin{equation}
		\begin{aligned}
			\dd_k & = \s_k^{(i)} - \x_k \\
			\s_k &\in \argmin_{\s \in \OM^{(i)}} \ps{\nabla f(\x_k)}{\s^{(i)}}
		\end{aligned}
	\end{equation}
	for a random index $i \in \irg 1n$. \\	
	In \cite{lacoste2013block}, a convergence rate of
	\begin{equation} \label{fwblockcr}
		\mathbb{E}[f(x_k)] - f^* \leq \frac{\imbb{2Km}}{k + 2m}
	\end{equation}	
	is proved, for $K  = h_0 + \kappa_f^{\otimes}$, with $\kappa_f^{\otimes}$ the product domain curvature constant, defined as  $\kappa_f^{\otimes} = \sum \kappa_f^{\otimes, i}$ where $\kappa_f^{\otimes, i}$ are the curvature constants of the objective \imbb{fixing the blocks outside} $\OM^{(i)}$:
	\begin{equation} \label{eq:cfo}
		\kappa_f^{\otimes, i} := 
\imbb{			\sup\left\{
			2{\textstyle{ \frac{f( \x+\alpha \dd^{(i)}) - f(\x)	- \alpha \ps{\nabla f(\x)}{\dd^{(i)}}}{\alpha^2}}}: \dd \in \OM-\x,\, \x\in \OM, \, \alpha \in (0, 1]
			\right\} \, .}
	\end{equation}
	An asynchronous and parallel generalization for this method was given in \cite{wang2016parallel}. This version assumes that a cloud oracle 
	is available, modeling a set of worker nodes each sending information to a server at different times. This information consists of an index $i$ and 
	\imbb{the following LMO on}  $\OM^{(i)}$:
	\begin{equation}
		\s_{(i)} \in \argmin\limits_{\s \in \OM^{(i)}} \ps{\nabla f(\x_{\widetilde{k}})}{\s^{(i)}} \, .
	\end{equation}
	The algorithm is  called asynchronous because $\widetilde{k}$ can be smaller than $k$, modeling a delay in the information sent by the node. Once the server has collected a minibatch $S$ of $\tau$ distinct indexes (overwriting repetitions), the descent direction is defined as
	\begin{equation}
		\dd_k = \sum_{i \in S} \s^{(i)}_{\imbb{(i)}} \, ,
	\end{equation} 
	If the indices sent by the nodes are i.i.d., then under suitable assumptions on the delay, a convergence rate of
	\begin{equation}
		\mathbb{E}[f(\x_k)] - f^* \leq \frac{2mK_{\tau}}{\tau^2k + 2m}
	\end{equation}
	can be proved, where $K_{\tau} = m\kappa_{f, \tau}^{\otimes}(1 + \delta) + h_0$ for $\delta$ depending on the delay error, with $\kappa_{f, \tau}^{\otimes}$ the average curvature constant in a minibatch keeping all the components not in the minibatch fixed. 
	
	In \cite{osokin2016minding}, several improvements are proposed for the BCFW, including an adaptive criterion to prioritize blocks based on their FW gap, and block coordinate versions of the AFW and the PFW variants.
	
	In \cite{shah2015multi}, a multi plane BCFW approach is proposed in the specific case of the structured SVM, based on caching supporting planes in the primal, corresponding to block linear minimizers in the dual. In \cite{berrada2018deep}, the duality for structured SVM between BCFW and stochastic subgradient descent is exploited to define a learning rate schedule for neural networks based only on one hyper parameter. The block coordinate approach is extended to the generalized FW in \cite{beck2015cyclic}, with coordinates however picked in a cyclic order.

	\subsection{Variants for the min-norm point problem}\label{mnpsec}
	Consider the min-norm point (MNP) problem
	\begin{equation}\label{MNP}
		\min_{\x \in \OM} {\|\x\|}_{*} \, ,
	\end{equation}
	with $\OM$ a closed convex subset of $\R^n$ and ${\n{\cdot}}_{*}$ a norm on $\R^n$. In \cite{wolfe1976finding}, a FW variant is introduced to solve the problem when $\OM$ is a polytope and ${\n{\cdot}}_*$ is the standard Euclidean norm $\n{\cdot}$. Similarly to the variants introduced in Section~\ref{s:var}, it generates a sequence of active sets $\{A_k\}$ with $ \s_k\in A_{k + 1}$. At the step $k$ the norm is minimized on the affine hull $\aff(A_k)$ of the current active set $A_k$, that is 	\begin{equation}
		\vv_k = \argmin_{\y \in \aff(A_k)} \n{\y} \, .
	\end{equation} 
	The descent direction $\dd_k$ is then defined as
	\begin{equation}
		\dd_k = \vv_k - \x_k\, ,
	\end{equation}
	and the stepsize is given by a tailored linesearch that allows to remove some of the atoms in the set $A_k$ (see, e.g. \cite{lacoste2015global,wolfe1976finding}). Whenever $\x_{k + 1}$ is in the relative interior of $\conv(A_k)$, the FW vertex is added to the active set (that is, $\s_k \in A_{k + 1}$). Otherwise, at least one of the vertices not appearing in a convex representation of $\x_k$ is removed. This scheme converges linearly when applied to generic smooth strongly convex objectives (see, e.g., \cite{lacoste2015global}). 
	
	In \cite{harchaoui2015conditional}, a FW variant is proposed for minimum norm problems of the form 
	\begin{equation}\label{eq:mnp}
		\min \{{\n{\x}}_* : f(\x) \leq 0, \,  \x \in K \}
	\end{equation}
	with $K$ a convex cone, $f$ convex with $L$-Lipschitz gradient. In particular, the optimization domain is $\OM = \{\x\in \R^n : f(\x) \leq 0\} \cap K$. The technique proposed in the article applies the standard FW method to the problems 
	$$\min \{f(\x) : {\n{\x}}_* \leq \delta_k, \, \x \in K\} \, ,$$ 
	with $\{\delta_k\}$ an increasing sequence convergent to the optimal value $\bar{\delta}$ of the problem~\eqref{eq:mnp}. Let $\OM(\delta) = \{\x\in \R^n : {\n{\x}}_* \leq \delta\} \cap K $ for $\delta \geq 0$, and let 
	$$\imbb{\tx{LM}}(\rr) \in \argmin\limits_{\imbb{\x\in \OM(1)}} \ps{\rr}{\x} \, ,$$
	so that by homogeneity for every $k$ the linear minimization oracle on $C(\delta_k)$ is given by
	\begin{equation}
		\LMO_{\OM(\delta_k)}(\rr) = \delta_k \imbb{\tx{LM}}(\rr) \, .
	\end{equation}
	For every $k$, applying the FW method with suitable stopping conditions an approximate minimizer $\x_k$ of $f( \x)$ over $\OM(\delta_k)$ is generated, with an associated lower bound on the objective, \imbb{an affine function in $\y$}:
	\begin{equation}
		f_{k}(\y) := f(\x_k) + \ps{\nabla f(\x_k)}{(\y - \x_k)} \, .
	\end{equation}
	Then the function 
	\begin{equation}
		\ell_{k}(\delta) := \imbb{\min_{\y \in \OM(\delta)} f_k(\y) } = f_{k}(\delta \imbb{\tx{LM}(\g_k)) \quad\tx{with } \g_k= \nabla f(\x_k)}
	\end{equation}
	is decreasing \imbb{and affine} in $\delta$ and satisfies
	\begin{equation}
		\ell_{k}(\delta) = \min_{\y \in \OM(\delta)} f_k(\y) \leq F(\delta) : =\min_{\y \in \OM(\delta)} f(\y)  \, . 
	\end{equation}
	Therefore, for
	$$\bar{\ell}_k(\delta) = \max_{i \in [1:k]} \ell_i(\delta) \leq F(\delta) $$
	the quantity $\delta_{k + 1}$ can be defined as $\min \{\delta \geq 0 : \imbb{\bar\ell_{k}(\delta)} \leq 0 \}$, \imbb{hence} $F(\delta_{k + 1}) \geq 0$. A complexity bound of ${\mathcal O}(\frac{1}{\eps} \ln(\frac{1}{\eps}))$ was given to achieve precision $\eps$ applying this method, with ${\mathcal O}(1/\eps)$ iterations per subproblem and length of the sequence $\{\delta_k\}$ at most ${\mathcal O}(\ln(1/\eps))$  (see \cite[Theorem 2]{harchaoui2015conditional} for details).
	
	\subsection{Variants for optimization over the trace norm ball}
	
	The FW method has found many applications for optimization problems over the trace norm ball. In this case, as explained in Example \ref{s:mcp}, linear optimization can be obtained by computing the top left and right singular vectors of the matrix $-\nabla f(\Xb_k)$, an operation referred to as 1-SVD (see \cite{allen2017linear}) . 
	
	In the work \cite{freund2017extended}, the FDFW is applied to the matrix completion problem \eqref{eq:mc}, thus generating a sequence of matrices $\{\Xb_k\}$ with ${\n{\Xb_k}}_* \leq \delta$ for every $k$. The method can be implemented efficiently exploiting the fact that for $\Xb$ on the boundary of the nuclear norm ball, there is a simple expression for the face $\F(\Xb)$. For $\Xb \in \mathbb{R}^{m \times n}$ with $\rank(\Xb) = k$ let $\Ub \Db \Vb^{\intercal}$ be the thin SVD of $\Xb$, so that $\Db \in \mathbb{R}^{k \times k}$ is the diagonal matrix of non zero singolar values for $\Xb$, with corresponding left and right singular vectors in the columns of $\Ub \in \mathbb{R}^{m \times k}$ and $\Vb \in \mathbb{R}^{n \times k}$ respectively.  If ${\n{\Xb}}_* = \delta$ then the minimal face of the domain containing $\Xb$ is the set
	\begin{equation}
		\F(\Xb) = \{\Xb \in \mathbb{R}^{m \times n} : \Xb = \Ub \Mb \Vb^{\intercal} \tx{ for } \imbb{\Mb=\Mb\T \ \psd \tx{ with }}
		{\n{\Mb}}_{*} = \delta \} \, .
	\end{equation}
	It is not difficult to see that we have $\rank(\Xb_k) \leq k + 1$ for every $k \in \mathbb{N}$, as well. Furthermore, the thin SVD of the current iterate $\Xb_k$ can be updated efficiently both after FW steps and after in face steps. The convergence rate of the FDFW in this setting is still ${\mathcal O}(1/k)$. \\ 
	In the recent work \cite{wang2020frank}, an unbounded variant of the FW method is applied to solve a generalized version of the trace norm ball optimization problem:
	\begin{equation}
		\min_{\Xb \in \mathbb{R}^{m \times n}}\{f(\Xb) : {\n{\Pb\Xb\Qb}}_* \leq \delta \}
	\end{equation}
	with $\Pb, \Qb $ singular matrices. The main idea of the method is to decompose the domain in the sum $S + T$ between the kernel $T$ of the linear function \mbox{$\varphi_{\Pb, \Qb}(\Xb)= \Pb\Xb\Qb$} and a bounded set $S \subset T^{\perp}$. Then gradient descent steps in the unbounded component $T$ are alternated to FW steps in the bounded component $S$. The authors apply this approach to the generalized LASSO as well, using the AFW for the bounded component. 
	
	In \cite{allen2017linear}, a variant of the classic FW using $k$-SVD (computing the top $k$ left and right singular vectors for the SVD) is introduced, and it is proved that it converges linearly for strongly convex objectives when the solution has rank at most $k$.  In \cite{mu2016scalable}, the FW step is combined with a proximal gradient step for a quadratic problem on the product of the nuclear norm ball with the $\ell_1$ ball. Approaches using an equivalent formulation on the spectrahedron introduced in \cite{jaggi2010simple} are analyzed in \cite{ding2020spectral,garber2019linear}.  
	
	\section{Conclusions}
	
	While the concept of the FW method is quite easy to understand, its advantages, witnessed by a multitude of related work, may not be apparent to someone not closely familiar with the subject. Therefore we  considered, in Section \ref{Sec:examples}, several motivating applications, ranging from classic optimization to more recent machine learning problems. As in any line search-based method, the proper choice of stepsize is an important ingredient to achieve satisfactory performance. In Section \ref{s:stp}, we review several options for stepsizes in first order methods, which are closely related both to the theoretical analysis as well as to practical implementation issues, guaranteeing fast convergence. This scope was investigated in more detail in Section \ref{s:cFW} covering main results about the FW method and its most popular variants, including the ${\mathcal O}(1/k)$ convergence rate for convex objectives, affine invariance, the sparse approximation property, and support identification. The account is complemented by a report on recent progress in improving on the ${\mathcal O}(1/k)$ convergence rate in Section~\ref{s:ir}. Versatility and  efficiency of this approach was also illustrated in the final Section \ref{s:E} describing present recent FW variants fitting different optimization frameworks and computational environments, in particular block coordinate, distributed, accelerated, and trace norm optimization. For sure many other interesting and relevant aspects of FW and friends could not find their way into this review because of space and time limitations, but the authors hope to have convinced readers that FW merits a consideration even by non-experts in first-order optimization.
\vskip 1.5em	
	
	\bibliographystyle{spmpsci}      
	\bibliography{review} 
	
\end{document}